\documentclass{amsart}
\usepackage{array}
\usepackage{amsmath}
\usepackage{latexsym}
\usepackage{xypic}
\usepackage{mathrsfs}
\usepackage{enumerate}

\newtheorem{theorem}{Theorem}[section]
\newtheorem{lemma}[theorem]{Lemma}
\newtheorem{thm}[theorem]{Theorem}

\theoremstyle{definition}
\newtheorem{definition}[theorem]{Definition}
\newtheorem{example}[theorem]{Example}

\theoremstyle{remark}
\newtheorem{remark}[theorem]{Remark}
\newtheorem{rmk}[theorem]{Remark}
\newtheorem{prop}[theorem]{Proposition}
\newtheorem{cor}[theorem]{Corollary}
\newtheorem{err}[theorem]{Erratum}

\numberwithin{equation}{section}

\begin{document}
\title{ Gorenstein stable log surfaces with $(K_X+\Lambda)^2=p_g(X,\Lambda)-1$}
\author{Jingshan Chen}
\address{Yau Mathematical Sciences Center, Tsinghua University, Beijing 100084}
\curraddr{Yau Mathematical Sciences Center, Tsinghua University,Beijing 100084}
\email{chjingsh@mail.tsinghua.edu.cn}





\begin{abstract}
In this paper, we will give a complete classification of Gorenstein stable log surfaces $(X,\Lambda)$ with $(K_X+\Lambda)^2=p_g(X,\Lambda)-1$, where $p_g(X,\Lambda):=h^0(X,K_X+\Lambda)$.
In particular, we classify Gorenstein stable surfaces with $K_X^2=p_g-1$.
\end{abstract}

\maketitle

\section{Introduction}

KSBA stable (log) surfaces are the two-dimensional analogues of stable (pointed) curves. They are the fundamental objects in  compactifying  the moduli spaces of smooth surfaces of general type.

In general, stable (log) surfaces are difficult to classify. We may first focus on those with global index $I=1$, i.e. $K_X$ (resp. $K_X+\Lambda$) being Cartier. By an abuse of notation as in \cite{LR13}, they are called Gorenstein stable (log) surfaces.
In \cite{LR13}, Liu and Rollenske give several inequalities for numerical invariants of Gorenstein stable log surfaces. One important inequality among them is the stable log Noether inequality $(K_X+\Lambda)^2\ge p_g(X,\Lambda)-2$
(see \cite[Thm 4.1]{LR13}). This can be rephrased as $\Delta(X,K_X+\Lambda)\ge 0 $, where $\Delta$ is Fujita's $\Delta$-genus defined as $\Delta(X,\mathcal{L})=\mathcal{L}^{\mathrm{dim}X}-h^0(X,\mathcal{L})+\mathrm{dim}X$.

Gorenstein stable log surfaces with $\Delta(X,K_X+\Lambda)=0$ have been described in \cite{LR13}.
Normal Gorenstein stable log surfaces with $\Delta(X,K_X+\Lambda)=1$ have been classified in \cite{Chen18}.
Here we continue to classify non-normal Gorenstein stable log surfaces with $\Delta(X,K_X+\Lambda)=1$.
The main results are as follows.

\begin{theorem}
[irreducible case]
Let $(X,\Lambda)$ be an irreducible non-normal stable log surface with $K_X+\Lambda$ Cartier and $\Delta(X,K_X+\Lambda)=1$.
Let $\pi\colon \bar{X}\to X$ be the normalization map and let $\bar{D}$ and $\bar{\Lambda}$ be the pre-images of the non-normal locus $D$ and $\Lambda$.

Then
\begin{itemize}
\item either $\Delta(\bar{X},\pi^*(K_X+\Lambda))=1$ and $(X,\Lambda)$ is as in Thm \ref{delta1,1},
\item or $\Delta(\bar{X},\pi^*(K_X+\Lambda))=0$ and $(\bar{X},\bar{\Lambda}+\bar{D})$ is as in Thm \ref{delta1,0}.
\end{itemize}

In particular, if $\Lambda=0$, then $X$ is one of the followings:
\begin{itemize}
	\item $X$ is a double cover of $\mathbb{P}^2$. The branch curve is $2C+B$, where $C$ and $B$ are reduced curves of degree $4-k$ and $2k$ respectively. $k=2,3$.  ($p_g(X)=3$)
	\item $X$ is obtained from a log surface $(\bar{X},\bar{D})$ by gluing the 2-section $\bar{D}$,   where $(\bar{X},\bar{D})$ is a normal Gorenstein stable log surface as in Thm \ref{delta-genus-one} (2) or (3).  ($p_g(X)=2$)
	\item $X$ is obtained from $\mathbb{P}^2$ by gluing a quartic curve $\bar{D}$. ($p_g(X)=2$)
	\item  $X$ is obtained from a quadric in $\mathbb{P}^3$ by gluing a curve $\bar{D}$, where $\bar{D}\in |\mathcal{O}_{\bar{X}}(3)|$. ($p_g(X)=3$)
\end{itemize}

\end{theorem}

\begin{theorem}[reducible case]
Let $(X,\Lambda)$ be a reducible stable log surface with $K_X+\Lambda$ Cartier and $\Delta(X,K_X+\Lambda)=1$.
Write $X=\bigcup X_i$, where $X_i$ is an irreducible component.

Then  $\Delta(X_i,(K_X+\Lambda)|_{X_i})=1$ or $0$ for each component $X_i$,
$X$ has a unique minimal connected component $U$ such that $\Delta(U,(K_X+\Lambda)|_{U})=1$, $X\setminus U$ is composed of several trees of surfaces $T_j$'s with $\Delta(T_j,(K_X+\Lambda)|_{T_j})=0$ and $X$ is glued by $U$ and $T_j$'s along lines, i.e. $X=U\cup \bigcup T_j$ with each $U\cap T_j$ a line with respect to $K_X+\Lambda$. $U$ is as described in Thm \ref{nonnormdelta-1}.

In particular, if $\Lambda=0$, $X$ is a union of two $\mathbb{P}^2$'s glued along a quartic curve on each $\mathbb{P}^2$.  ($p_g(X)=3$)

\end{theorem}

We give a brief account of each section. In \textsection \ref{prelim}, we review the definition and some facts about Gorenstein stable log surfaces.
In \textsection \ref{zero-Delta-genus}, we recall the definition of Fujita's $\Delta$-genus and include some results about normal Gorenstein stable log surfaces $(X,\Lambda)$ with
$\Delta(X,K_X+\Lambda)=0\,, 1$. 
In \textsection \ref{non-normal}, we discuss non-normal Gorenstein stable log surfaces.
In \textsection \ref{nnorm}, we deal with the case that $X$ is  non-normal and irreducible.
In \textsection \ref{reducible-stable}, we work on the case that $X$ is reducible.
Finally we describe Gorenstein stable surfaces with $K_X^2= p_g-1$.

\subsection*{Acknowledgements:}
I am grateful to Prof. Jinxing Cai and Prof. Wenfei Liu for their instructions. I would also thank the anonymous referee for helpful advices and suggestions.

\subsection{Notations and conventions}
We work exclusively with schemes of finite type over the complex numbers $\mathbb{C}$.
\begin{itemize}
\item[.] A surface is a connected reduced projective  scheme of pure dimension two.
\item[.] By an abuse of notation, we sometimes do not distinguish between a Cartier divisor $D$ and its associated invertible sheaf $\mathcal{O}_X(D)$.
\item[.] $\Sigma_d$ denotes a Hirzebruch surface, which admits a $\mathbb{P}^1$ fibration over $\mathbb{P}^1$. We denote by $\Gamma$ a fiber and by $\Delta_0$ the 1-section  whose self-intersection is $-d$.
\item[.] We use '$\equiv$' to denote linear equivalent relation of divisors.
\item[.] If $D$ is a Cartier divisor on $X$, then we denote by $\Phi_{|D|}\colon X\dashrightarrow \mathbb{P}:=|D|^*$  the rational map defined by the linear system $|D|$.
\item[.] A line $l$ on a variety $X$ with respect to $\mathcal{O}_X(1)$ is a rational curve such that $l\cdot \mathcal{O}_X(1)=1$.

\item[.] We say a connected reducible surface $X=\bigcup \limits_{i=1}^{n} X_i$ is a \emph{string} of surfaces, if each $X_i$ is irreducible, $X_i\cap X_j$ is either an irreducible curve or a set of points(which may be empty) for all $j\not=i$, and each $X_i$ is connected to at most two other surfaces in codimension one.

\item[.] We say a connected reducible surface $X=\bigcup \limits_{i=1}^{n} X_i$, ($n\ge3$) is a \emph{cycle} of surfaces, if each $X_i$ is irreducible, $X_i\cap X_j$ is either an irreducible curve or a set of points(which may be empty) for all $j\not=i$ and there is an indexing such that $X_i\cap X_{i+1}$, $i=1,...,n-1$ and $X_n\cap X_1$ are all irreducible curves.

A reducible surface $X_1\cup X_2$ is also called  a \emph{cycle} of surfaces, if each $X_i$ is irreducible and $X_1\cap X_2$ consists of two irreducible curves.

\item[.] We say a connected reducible surface $X=\bigcup \limits_{i=1}^{n} X_i$ is a \emph{tree} of surfaces, if $X_i\cap X_j$ is either an irreducible curve or a set of points(which may be empty) for all $j\not=i$ and there is an indexing such that $\bigcup\limits_{i=1}^{j}X_i\cap X_{j+1}$ is an irreducible curve for each $j$.
\end{itemize}

\section{Preliminaries}
\label{prelim}

In this section we include some notions and definitions from \cite{LR13} and \cite[\textsection~5.1--5.3]{KollarSMMP}.

Let $X$ be a demi-normal surface, i.e. $X$ satisfies $S_2$ and has at worst an ordinary double point at any generic point of codimension 1.
Denote by $\pi\colon  \bar X \to X$ the normalization map of $X$. The conductor ideal
$ \mathrm{\mathcal{H}om}_{\mathcal{O}_X}(\pi_*\mathcal{O}_{\bar{X}}, \mathcal{O}_X)$
is an ideal sheaf both on $X$ and $\bar{X}$ and hence defines subschemes
$D\subset X \text{ and } \bar D\subset \bar X$,
both reduced and of pure codimension 1. $D$ is referred to as the non-normal locus of $X$.

Let $\Lambda$ be a reduced curve on $X$ whose support does not contain any irreducible component of $D$. Then the strict transform $\bar \Lambda$ in $\bar{X}$ is well defined.

Such a pair $(X, \Lambda)$ as above is called a \emph{log surface}; $\Lambda$ is called the (reduced) boundary.
We have $\pi^*(K_X+\Lambda)=K_{\bar{X}}+\bar D+\bar \Lambda$ . 
\begin{definition}
	\label{defin: slc}
A log surface $(X,\Lambda)$ is said to have \emph{semi-log-canonical (slc)}  singularities if it satisfies the following conditions:
\begin{enumerate}
 \item $K_X + \Lambda$ is $\mathbb{Q}$-Cartier, i.e. $m(K_X+\Lambda)$ is Cartier for some $m\in\mathbb{Z}^{>0}$; the minimal such $m$ is called the (global) index of $(X,\Lambda)$.
\item The pair $(\bar X, \bar \Lambda+\bar D)$ has log-canonical singularities.
\end{enumerate}

The pair $(X,\Lambda)$ is called a stable log surface if in addition $K_X+\Lambda$ is ample.  A stable surface is a stable log surface with empty boundary.

By an abuse of notation as in \cite{LR13} $(X, \Lambda)$ is called a Gorenstein stable log surface if  the index is equal to one, i.e. $K_X+\Lambda$ is an ample Cartier divisor.
\end{definition}

\subsection{Koll{\'a}r's gluing principle}
 Since $X$ has at most double points in codimension one the map $\pi|_{\bar{D}}\colon \bar D \to D$ is generically a double cover and thus  induce a rational involution on $\bar D$. Normalizing $\bar D$ we get an honest involution $\tau\colon \bar D^\nu\to \bar D^\nu$ such that $D^\nu = \bar D^\nu/\tau$ and the different $\mathrm{Diff}_{\bar D^\nu}(\Lambda)$ is $\tau$-invariant(for the definition of the \emph{different} see \cite[\textsection~5.11]{KollarSMMP}).

\begin{thm}[{\cite[Thm.~5.13]{KollarSMMP}}]\label{thm: triple}
	Associating to a log-surface $(X, \Lambda)$ the triple $(\bar X, \bar D+\bar \Lambda, \tau\colon \bar D^\nu\to \bar D^\nu)$ induces a one-to-one correspondence
	\[
	\left\{ \text{\begin{minipage}{.12\textwidth}
		\begin{center}
		stable log surfaces  $(X, \Lambda)$
		\end{center}
		\end{minipage}}
	\right\} \leftrightarrow
	\left\{ (\bar X, \bar D+\bar \Lambda, \tau)\left|\,\text{
		\begin{minipage}{.37\textwidth}
		$(\bar X, \bar D+\bar \Lambda)$ a log-canonical pair with
		$K_{\bar X}+\bar D+\bar \Lambda$ ample, \\
		$\tau\colon \bar D^\nu\to \bar D^\nu$  an involution s.t.\
		$\mathrm{Diff}_{\bar D^\nu}(\bar\Lambda)$ is $\tau$-invariant.
		\end{minipage}}\right.
	\right\}.
	\]
\end{thm}
\subsection{Classification of Gorenstein slc singularities}\label{clssofslc}
Let $x\in (X,\Lambda)$ be a germ of a Gorenstein slc singularity, then it is one of the following(see \cite[Ch.~4]{Kollar-Mori},  \cite[Sect.~3.3]{KollarSMMP}, \cite{kollar12} and \cite{LR13}):
\begin{enumerate}[1)]
	\item  $\Lambda=0$, $X$ is non-normal and $x$ is a general point of the non-normal locus $D$. In this case $\bar{X}$ is smooth and has two components, and $\bar{D}$ is a sum of two curves.
	\item  $\Lambda=0$, $X$ is non-normal and $x$ is a pinch point whose local model in $\mathbb{A}^3$ is given by the equation $x^2+yz^2=0$. In this case $\bar{X}$ is smooth and has one component, and $\bar{D}$ is an irreducible curve.
	\item $\Lambda=0$, $X$ is normal and $x$ is a canonical singularity of $X$, i.e an A-D-E surface singularity.
	\item $\Lambda=0$, $X$ is normal and $x$ is a simple elliptic singularity or a cusp singularity of $X$. The exceptional curve of $x$ is a smooth elliptic curve, a nodal rational curve or a cycle of smooth rational curves.
	\item $\Lambda\not =0$, $X$ is normal and smooth, and $x$ is a node of $\Lambda$.
	\item $\Lambda\not =0$, $X$ is normal, $x$ is a cyclic quotient singularity of $X$, and $\Lambda$ is a generic hyperplane section of $x$. In the minimal log resolution the dual graph of the exceptional curves is
	\[\bullet   \ {-}\ c_1  \ -\ \cdots  \ - \ c_n  \ {-}
	\ \bullet \qquad (c_i\geq2),\]
	where $c_i$ represents a smooth rational curve of self-intersection $-c_i$ and each $\bullet$ represents a (local) component of the strict transform of $\Lambda$.
	\item $\Lambda=0$, $X$ is non-normal and its normalization $(\bar{X},\bar{D})$ has $k\ge 1$ components with each component $(\bar{X}_i,\bar{D}_i)$ has a singularity $\bar{x}_i$ of type 5) or 6). If $k\ge2$, $X$ is a cycle of surfaces. If $k=1$, $\bar{D}=\bar{D}_1+\bar{D}_2$ and $\tau$ permutes $\bar{D}_1$ and $\bar{D}_2$.
	\item $\Lambda\not=0$, $X$ is non-normal. $X$ is a string of surfaces whose normalization $(\bar{X},\bar{D}+\bar{\Lambda})$ has $k\ge 2$ components  $(\bar{X}_1,\bar{\Lambda}_1+\bar{D}_1)$, $(\bar{X}_i,\bar{D}_i)$, $i=2,...,k-1$, and $(\bar{X}_k,\bar{\Lambda}_k+\bar{D}_k)$. Each component has a singularity $\bar{x}_i$ of type 5) or 6).
\end{enumerate}

\section{Normal Gorenstein stable log surfaces with small $\Delta$-genus}
\label{zero-Delta-genus}
Let $X$ be a variety and $\mathcal{L}$ be an ample line bundle on it.
Fujita introduced several invariants for such polarized varieties. One important of them is the
$\Delta$-genus $\Delta(X,\mathcal{L}):=\mathcal{L}^{\dim X}-h^0(X,\mathcal{L})+\dim X$. Some fundamental facts about $\Delta$-genus are included in the following theorem: (see \cite{Fuj75,Fuj90}).

\begin{thm}\label{Fujita-Delta-genus}
 	Let $(X,\mathcal{L})$ be a polarized variety.
 	Then
 	\begin{displaymath}
 	\Delta(X, \mathcal{L})\ge 0,
 	\end{displaymath}
 	and
 	\begin{displaymath}
 	dim\,Bs|\mathcal{L}|< \Delta(X, \mathcal{L}),
 	\end{displaymath}
 	where $Bs|\mathcal{L}|$ is the base locus of $|\mathcal{L}|$.
 	
 	Moreover, $\mathcal{L}$ is very ample if $\Delta(X, \mathcal{L})= 0$.
\end{thm}

For normal irreducible Gorenstein stable log surfaces with $\Delta(X,K_X+\Lambda)= 0,1$, we have the following results:(see \cite{Chen18})

\begin{cor}\label{gepg-2}
Let $(X,\Lambda)$ be a normal irreducible Gorenstein stable log surface. Then
$\Delta(X,K_X+\Lambda)\ge 0$.

Moreover if '=' holds, then
$(X,\Lambda)$ is one of the followings:
\begin{itemize}

\item[(i)] $X$ is $\mathbb{P}^{2}$, $\mathcal{O}_X(K_X+\Lambda) =\mathcal{O}_{\mathbb{P}^{2}}(1)$ and $\Lambda\in|\mathcal{O}_{\mathbb{P}^{2}}(4)|$; 
\item[(ii)] $X$ is $\mathbb{P}^{2}$,  $\mathcal{O}_X(K_X+\Lambda)=\mathcal{O}_{\mathbb{P}^{2}}(2)$ and $\Lambda\in|\mathcal{O}_{\mathbb{P}^{2}}(5)|$;
\item[(iii)] $X$ is $\Sigma_d$, $K_X+\Lambda\equiv \Delta_0+\frac{N+d-1}{2}\Gamma$ and  $\Lambda\in|3\Delta_0+\frac{N+3d+3}{2}\Gamma|$; ($N-d-3\ge0$ is an even number);

\item[(iv)] $X$ is a singular quadric $C_2$ in $\mathbb{P}^{3}$, $\mathcal{O}_X(K_X+\Lambda)=\mathcal{O}_{C_2}(1)=\mathcal{O}_{\mathbb{P}^{2}}(1)|_{C_2}$ and  $\Lambda\in|\mathcal{O}_{C_2}(3)|$;
\item[(v)]  $X$ is a cone $C_{N-1}\hookrightarrow \mathbb{P}^{N}$, $\mathcal{O}_X(K_X+\Lambda)=\mathcal{O}_{\mathbb{P}^{N}}(1)|_{C_{N-1}}$ and the proper transformation $\bar{\Lambda}$ in the minimal resolution $\Sigma_{N-1}$ is linearly equivalent to $2\Delta_0+2N\Gamma$. ($N\ge 4$)
\end{itemize}
\end{cor}

\begin{theorem}\label{delta-genus-one}
Let $(X,\Lambda)$ be a normal Gorenstein stable log surface with $(K_X+\Lambda)^2= p_g(X,\Lambda)-1$.

Then $(X,\Lambda)$ is one of the followings:
\begin{enumerate}[(1)]
\item $\Lambda=0$, and $X$ is canonically embedded as a hypersurface of degree 10 in the smooth locus of $\mathbb{P}(1,1,2,5)$. In this case $|K_X|$ is composite with a pencil of genus $2$ curves.

\item $X$ is a (possibly singular) Del Pezzo surface of degree 1, namely $X$ has at most canonical singularities, $-K_X$ is ample and $K_X^2=1$. The curve $\Lambda$ belongs to the system $|-2K_X|$, and $p_a(\Lambda)$=2.

\item  
$X$ is a singular surface with an elliptic singularity, $-K_X$ ample and $K_X^2=1$.  The curve $\Lambda$ belongs to the system $|-2K_X|$, and $p_a(\Lambda)$=2.

\item $X$ is obtained from $\tilde{X}$ by contracting a $(-N)$ curve $G$, where $N:=p_g(X,\Lambda)-1\ge2$ and $\tilde{X}$ is an elliptic surface (possible singular) with $G$ as the rational zero section. Every elliptic fiber of $\tilde{X}$ is irreducible. 
$\Lambda$ is the image of  a sum of two different elliptic fibers  which admit at worst $A_n$ type singularities.

\item  $X$ is a double cover of $\mathbb{P}^2$. $\Phi_{|K_X+\Lambda|}\colon X\to \mathbb{P}^2$ is the double covering map.  The branch curve $B\in |\mathcal{O}_{\mathbb{P}^2}(2k)|$ is a reduced curve which admits curve singularities of lc double-covering type, and $\Lambda\in |{\Phi_{|K_X+\Lambda|}}^*\mathcal{O}_{\mathbb{P}^2}(4-k)|$.  ($p_g(X,\Lambda)=3$, and $k=2,3,4$)
\item  $X$ is a quadric in $\mathbb{P}^3$, and $\Lambda\in |\mathcal{O}_{\mathbb{P}^3}(4)|_X|$. ($p_g(X,\Lambda)=6$)
\item  $X$ is $\mathbb{P}^2$ blown up at $k$ points (possible infinitely near), and $\Lambda\in |-2K_X|$. ($p_g(X,\Lambda)=10-k$, and $k=0,1,...,7$)
\item  $X$ is  a cone over an elliptic curve of degree $N$ in $\mathbb{P}^{N-1}$, and $\Lambda\in |-2K_X|$. ($p_g(X,\Lambda)=N$)

\end{enumerate}

\end{theorem}

\begin{remark}
The cases (1)-(3) are in Thm 4.1 of \cite{Chen18} which is a result collected from \cite{FPR15a} and \cite{FPR15b}.  The case (4) is in Thm 4.4 \cite{Chen18} and the cases (5)-(8) are in Thm 4.3 of \cite{Chen18}.
In fact, the cases (2) and (3) are special cases of (4) when $N=1$.
Moreover, in the cases  (1)-(4)  $|K_X+\Lambda|$ is composite with a pencil and in the cases (5)-(8)  $|K_X+\Lambda|$ is not composite with a pencil.
\end{remark}
\begin{err}
	In \cite{Chen18} Theorem 1.1(1)  and Theorem 4.3(i), $k=1,2,3,4$ should be corrected by $k=2,3,4$. The case $k=1$ can be excluded. Since in this case $h^0(X,K_X+\Lambda)=h^0(\mathbb{P}^2,\mathcal{O}_{\mathbb{P}^2}(1))+h^0(\mathbb{P}^2,\mathcal{O}_{\mathbb{P}^2})=4$, which implies $\Delta(X,K_X+\Lambda)=0$, contradicting the hypothesis of the theorem.	
	This erratum applies to case (5) of Thm \ref{delta-genus-one}.
\end{err}

\section{Non-normal Gorenstein stable log surfaces}\label{non-normal}

Let $(X,\Lambda)$ be a Gorenstein stable log surface which is non-normal. $D$ is the non-normal locus of $X$.

\begin{thm}
Let $C\subset D$ be a subcurve of the non-normal locus of a Gorenstein stable log surface $(X,\Lambda)$ and $\nu\colon\bar{X}\to X$ be the partial normalization of $X$ along $C$, i.e. a normalization along $C$ and then a seminormalization.
Denote by $\bar C$ and $\bar{\Lambda}$ the proper transformation of $C$ and $\Lambda$.

Then $(\bar{X},\bar{\Lambda}+\bar C)$ is a Gorenstein stable log surface.
\end{thm}
\begin{proof}
	This is a direct result of Koll{\'a}r's gluing principle.
\end{proof}

The map $\nu|_{\bar C}\colon \bar C \to C$ is generically a double cover and thus  induces a rational involution $\tau$ on $\bar C$ and an honest involution on $\bar C^\nu$.

\begin{prop}[\protect{see \cite[Prop 2.6]{LR13} and \cite[Prop.~5.8]{KollarSMMP}}]\label{prop: descend section}
Let $(X, \Lambda)$ be a non-normal Gorenstein stable log surface and $\nu$, $C$, $\bar C$, $\tau$ are defined as above.

Then
 $\nu^*H^0(X, \omega_X(\Lambda))\subset H^0(\bar{X}, \omega_{\bar X}(\bar C+\bar\Lambda)))$ is the subspace of those sections $s$  whose residue in $H^0(\bar C^\nu, \omega_{\bar C^\nu}(\mathrm{Diff}_{\bar C^\nu}(\Lambda)))$ 
 is $\tau$-anti-invariant.
\end{prop}

\begin{rmk}\label{separateNonnormalCurve}
It is easy to see that $\nu^*H^0(X, \omega_X(\Lambda))$ can not separate the sheets of $\bar C$. Actually, $\Phi_{\nu^*H^0(X, \omega_X(\Lambda))}|_{\bar C}{}_\circ \tau=\Phi_{\nu^*H^0(X, \omega_X(\Lambda))}|_{\bar C}$ on $\bar C$.
\end{rmk}

\begin{rmk}[\protect{cf. \cite[Remark 2.7]{LR13}}]
\label{fibersection}
  Let $(X,\Lambda)$ be a reducible Gorenstein stable log surface such that $X=X_1\cup X_2$ and
$C:=X_1\cap X_2$ is the connecting curve (since $X$ is $S_2$, the isolated singularities of $X$ do not make a difference to the global sections of $\mathcal{O}_X(K_X+\Lambda)$ on $X$. We may drop the smooth assumption of $C$ as in \cite[Remark 2.7]{LR13}).
Let $\nu\colon \bar{X}=\bar{X_1}\sqcup \bar{X_2} \to X $ be the partial normalization of $X$ along $C$.
Write $\bar \Lambda=\bar{\Lambda}_1+\bar{\Lambda}_2$ and $\bar C=\bar{C}_1+\bar{C}_2$, where $\bar{\Lambda}_i$ and $\bar{C}_i$ are the components of $\bar\Lambda$ and $\bar{C}$ on $\bar{X_i}$.
We see that $\nu^*(K_X+\Lambda)|_{\bar{X}_i}=K_{\bar{X}_i}+\bar{C}_i+\bar{\Lambda}_i$ and $\bar{C}_i$ is birational to $C$.

The restriction map(which coincides with the residue map) induces
$$\mathcal{R}_{\bar{X}_i\to C}\colon H^0(\bar{X}_i,K_{\bar{X}_i}+\bar{C}_i+\bar{\Lambda}_i)\to H^0(\bar{C}_i,(K_{\bar{X}_i}+\bar{C}_i+\bar{\Lambda}_i)|_{\bar{C}_i})\cong H^0(C,K_C+\mathrm{Diff}_C(\Lambda)).$$
Then we have the following fiber product diagram of vector spaces:
\[ \xymatrix{
    H^0(X, K_X+\Lambda) \ar[r]\ar[d] & H^0(\bar{X}_1,K_{\bar{X}_1}+\bar{C}_1+\bar{\Lambda}_1)\ar[d]^{\mathcal{R}_{\bar{X}_1\to C}}\\
H^0(\bar{X}_2,K_{\bar{X}_2}+\bar{C}_2+\bar{\Lambda}_2)\ar[r]^{-\mathcal{R}_{\bar{X}_2\to C}}& H^0(C,K_C+\mathrm{Diff}_C(\Lambda))
   }.
\]
By an abuse of notation we may write
\[
H^0(X, K_X+\Delta) = H^0(\bar{X}_1,K_{\bar{X}_1}+\bar{C}_1+\bar{\Lambda}_1) \times_C H^0(\bar{X}_2,K_{\bar{X}_2}+\bar{C}_2+\bar{\Lambda}_2).
\]
Moreover, write $d_{\bar{X}_i\to C}(K_X+\Lambda):=\dim\mathcal{R}_{\bar{X}_i\to C}(H^0(\bar{X}_i,\nu^*(K_X+\Lambda)|_{\bar{X}_i}))$. We have
\begin{align*}
h^0(X, K_X+\Lambda)\le& h^0(\bar{X}_1,K_{\bar{X}_1}+\bar{C}_1+\bar{\Lambda}_1)+h^0(\bar{X}_2,K_{\bar{X}_2}+\bar{C}_2+\bar{\Lambda}_2)\\
&-\max\{d_{\bar{X}_1\to C}(K_X+\Lambda),d_{\bar{X}_2\to C}(K_X+\Lambda)\}.
\end{align*}
\end{rmk}

\section{irreducible Non-normal  Gorenstein stable surfaces with $\Delta(X,K_X+\Lambda)=1$}
\label{nnorm}

Let $(X,\Lambda)$ be an irreducible non-normal Gorenstein stable log surface
and $\bar{X}$, $\pi$, $D$, $\bar{D}$, $\tau$  defined as before.
Since $H^0(X,K_X+\Lambda)\cong \pi^*H^0(X,K_X+\Lambda) \subset H^0(\bar X,\pi^*(K_X+\Lambda))$, we have
$\Delta(X,K_X+\Lambda)\ge \Delta(\bar{X},\pi^*(K_X+\Lambda))$, and '=' holds if and only if $\pi^*H^0(X,K_X+\Lambda) = H^0(\bar X,\pi^*(K_X+\Lambda))$.

\begin{lemma}\label{nonnorm&irred}
Let $(X,\Lambda)$ be an irreducible non-normal Gorenstein stable log surface.

Then
$\Delta(X,K_X+\Lambda)\ge 1$.
\end{lemma}
\begin{proof}
We only need to show that the case $\Delta(X,K_X+\Lambda)=\Delta(\bar{X},\pi^*(K_X+\Lambda))=0$ does not occur.
Otherwise, $\Delta(\bar{X},\pi^*(K_X+\Lambda))=0$ implies that $\pi^*(K_X+\Lambda)$ is very ample by Thm \ref{Fujita-Delta-genus}. However,
$h^0(\bar{X},\pi^*(K_X+\Lambda))=h^0(X,K_X+\Lambda)$ implies $H^0(\bar{X},\pi^*(K_X+\Lambda))\cong \pi^*H^0(X,K_X+\Lambda)$, which can not separate sheets of $\bar{D}$ by Remark \ref{separateNonnormalCurve}. This contradicts to the very-ampleness of $\pi^*(K_X+\Lambda)$.

\end{proof}
\begin{theorem}\label{delta1,1}

Let $(X,\Lambda)$ be an irreducible non-normal Gorenstein stable log surface as before. Assume $\Delta(X,K_X+\Lambda)=\Delta(\bar{X},\pi^*(K_X+\Lambda))=1$.
Then
\begin{itemize}
  \item either $X$ is a double cover of $\mathbb{P}^2$ induced by $\Phi_{|K_{X}+\Lambda|}$.
      The branched curve is $2C+B\in |\mathcal{O}_{\mathbb{P}^2}(2m+2k)|$, where $C,B$ are reduced curves of degree $m,2k$. $\Lambda\in |\Phi_{|K_{X}+\Lambda|}^*\mathcal{O}_{\mathbb{P}^2}(4-k-m)|$.
      $k=2,3$. $0<m\le 4-k$; ($p_g(X,\Lambda)=3$)
  \item or $\Lambda=0$, $K_X^2=(K_{\bar{X}}+\bar D)^2=1$.
  $(\bar{X},\bar{D})$ is a normal Gorenstein stable log surface as in Thm \ref{delta-genus-one} (2) or (3).
  $\bar D$ is a 2-section with $p_a(\bar D)=2$. $\tau$ is induced by the associated double covering map $\bar D \to \mathbb{P}^1$. ($p_g(X)=2$)
\end{itemize}
\end{theorem}
\begin{proof}
  We see that $\pi^*H^0(X,K_X+\Lambda)=H^0(\bar{X},\pi^*(K_X+\Lambda))$ and $(\bar{X},\bar D+\bar{\Lambda})$ is a stable log surface as in Thm \ref{delta-genus-one}.

  First $\Phi_{|\pi^*(K_X+\Lambda)|}$ is not an embedding by Remark \ref{separateNonnormalCurve}. Second by Thm \ref{delta-genus-one} (1) $\Phi_{|\pi^*(K_X+\Lambda)|}$ is not composite with a pencil of genus 2 curves since $\bar D+\bar{\Lambda}\not= 0$.

  Now if $\Phi_{|\pi^*(K_X+\Lambda)|}$ is composite with a pencil of elliptic curves, $\bar{D}+\bar{\Lambda}$ is either a 2-section 
  or a sum of two elliptic fibers of the fibration map.
  For the former case, we see that $\bar{\Lambda}=0$ and $\bar{D}$ is a genus 2 curve by Remark \ref{separateNonnormalCurve}. Moreover, $(K_{\bar X}+\bar D)^2=1$.
  $\tau$ is induced by the associated double covering map $\bar{D} \to \mathbb{P}^1$.
  We show that the latter case does not happen. 
  We see that either $\bar{D}=F_1+F_2$ and $\bar{\Lambda}=0$, or $\bar{D}=F_1$ and $\bar{\Lambda}=F_2$, where $F_i$ is a fiber of the  fibration  map. Since $\Phi_{|\pi^*(K_X+\Lambda)|}$ restricting to $\bar{D}$ factors through the involution $\tau$, $\tau$ will fix $F_i$.
  Both cases can be excluded since the point $F_1\cap F_2$ or $F_1\cap\bar{\Lambda}$ on $\bar{X}$ can not be glued into a Gorenstein slc singularity on $X$ as in \textsection\ref{clssofslc}.

  If $\bar{X}$ is a double cover of $\mathbb{P}^2$, then $\bar{D}+\bar{\Lambda}\in |\Phi_{|K_{\bar{X}}+\bar{D}+\bar{\Lambda}|}^*\mathcal{O}_{\mathbb{P}^2}(4-k)|$, $k=2,3,4$ by Thm \ref{delta-genus-one} (5).
  $k=4$ is impossible, since $\bar{D}=0$ would imply $X$ is normal.
 Remark \ref{separateNonnormalCurve} indicates $\bar{D}=\Phi_{|K_{\bar{X}}+\bar{D}+\bar{\Lambda}|}^{-1}(C)$, where $C$ is a curve on $\mathbb{P}^2$. Denote by $m$ the degree of $C$.
 Then $\bar{\Lambda}\in |\Phi_{|K_{\bar{X}}+\bar{D}+\bar{\Lambda}|}^*\mathcal{O}_{\mathbb{P}^2}(4-k-m)|$.
 We have the following commutative diagram:
 \[
 \xymatrix{
 \bar{D}\ar@{^{(}->}[r]\ar[d] & \bar{X}\ar[d]^{\pi}\ar[rr]^{\Phi_{|K_{\bar{X}}+\bar{D}+\bar{\Lambda}|}} && \mathbb{P}^2\\
 D\ar@{^{(}->}[r] & X \ar[rru]_{\Phi_{|K_{X}+\Lambda|}}
 }.
 \]

 We see that $X$ is also a double cover of $\mathbb{P}^2$ induced by $\Phi_{|K_{X}+\Lambda|}$. The branch curve is $2C+B\in |\mathcal{O}_{\mathbb{P}^2}(2m+2k)|$. $\Lambda\in |\Phi_{|K_{X}+\Lambda|}^*\mathcal{O}_{\mathbb{P}^2}(4-k-m)|$.

\end{proof}

Next we consider those irreducible non-normal Gorenstein stable log surfaces $(X,\Lambda)$ with $\Delta(X,K_X+\Lambda)=1$ and $\Delta(\bar{X},\pi^*(K_X+\Lambda))=0$.

Since $(K_X+\Lambda)^2=\pi^*(K_X+\Lambda)^2$, we have $h^0(\bar{X},\pi^*(K_X+\Lambda))=h^0(X,K_X+\Lambda)+1$. Moreover, we have $\pi^*H^0(X,K_X+\Lambda)\cong H^0(X,K_X+\Lambda)$.
Hence $\pi^*H^0(X,K_X+\Lambda)\subset H^0(\bar{X},\pi^*(K_X+\Lambda))$ is of codimension one. Thus it has a base point $c\in \mathbb{P}(H^0(\bar{X},\pi^*(K_X+\Lambda))^*)$. We have the following commutative diagram:
  \[
  \xymatrix{
      \bar{X} \ar@{^{(}->}[rrr]^<(0.3){\Phi_{|\pi^*(K_X+\Lambda)|}}\ar[d]^{\pi} & & & \mathbb{P}(H^0(\bar{X},\pi^*(K_X+\Lambda))^*)\ar@{-->}[d]^{pr_c}\\
X\ar@{-->}[rrr]^<(0.3){\Phi_{|K_X+\Lambda|}}& & & \mathbb{P}(\pi^*H^0(X,K_X+\Lambda)^*)
  }.\]

We regard $\Phi_{|\pi^*(K_X+\Lambda)|}$ as an inclusion.
To describe $pr_c|_{\bar{X}}$, we need the following theorem (see \cite[Thm 2.1]{Ber06}):

\begin{thm}[Linear section theorem]\label{linearsectionthm}
Let $X$ be a non-degenerate
irreducible subvariety of $\mathbb{P}^r$ and $L$ be a linear subspace of $\mathbb{P}^r$ of dimension $s\le r$ such that $L\cap X$ is a 0-dimensional scheme $\zeta$. Then $\mathrm{length}\, \zeta \le \Delta(X,\mathcal{O}(1))+s+1$.
\end{thm}

\begin{cor}\label{lineIntersecting}
Let $X$ be a non-degenerate
irreducible subvariety of $\mathbb{P}^r$ with $\Delta(X,\mathcal{O}(1))=0$ and $L$ is a line intersecting $X$ along a 0-dimensional scheme $\zeta$. Then $\mathrm{length}\, \zeta\le 2$.
\end{cor}

\begin{cor}\label{proj-image}
  Let $X$ be a surface of minimal degree, i.e. a surface in $\mathbb{P}^N$  with $\Delta(X,\mathcal{O}(1))=0$. Let $c\in \mathbb{P}^N$ be a point outside $X$. Denote by $pr_c\colon \mathbb{P}^N\dashrightarrow \mathbb{P}^{N-1}$  the projection from $c$. Assume $pr_c|_X$ is birational and is not an isomorphism. Denote by $W$ the image of $pr_c|_X$.

  Then $N\ge 4$, $W$ is non-normal and $pr_c|_X\colon X\to W$ is the normalization map. The non-normal locus $D$ of $W$ is a line in $\mathbb{P}^{N-1}$. The pre-image $\bar{D}\colon = pr_c|_X^{-1}(D)$ is described as follows:
  \begin{itemize}
    \item if $X$ is $\mathbb{P}^2$ embedded in $\mathbb{P}^5$, then $\bar{D}\in |\mathcal{O}_{\mathbb{P}^2}(1)|$.
    \item if $X$ is a cone $C_{N-1}\hookrightarrow \mathbb{P}^N$,  then $\bar{D}$ is a sum of two rulings.
    \item if $X$ is $\Sigma_d\hookrightarrow \mathbb{P}^N$, then either $\bar{D}\in |\Delta_0+\Gamma|$, $N=d+3$, or $\bar{D}=\Delta_0$, $N=d+5$.
  \end{itemize}
\end{cor}
\begin{proof}
  If  $N=3$, then $pr_c|_X$ will be a morphism of degree 2. Hence $N\ge 4$.

  We show that $D$ is a line in $\mathbb{P}^{N-1}$ by contradiction hypothesis. Assume $D$ is not a line, then there are two points $p,q\in D$ such that the line $L_{p,q}$ passing through $p,q$ is not contained in $D$. Let $H_{p,q}\subset \mathbb{P}^N$ be the closure of the pre-image of $L_{p,q}$ with respect to $pr_c$. Then $H_{p,q}\cap X$ has length greater than $2L_{p,q}\cdot D=4$ which contradicts Thm \ref{linearsectionthm}. Therefore $D$ is  a line in $\mathbb{P}^{N-1}$.

  Finally, we describe $\bar{D}$. We first see that $\bar{D}\cdot \mathcal{O}_{X}(1)=2$ since $D$ is a line in $\mathbb{P}^{N-1}$.
  For the case where $X$ is $\mathbb{P}^2$ embedded in $\mathbb{P}^5$, we see that
  $\bar{D}\in |\mathcal{O}_{\mathbb{P}^2}(1)|$.
  For the case where $X$ is a cone $C_{N-1}$,
  we see that $\bar{D}$ is a sum of two rulings.
  For the case where $X$ is $\Sigma_d\hookrightarrow \mathbb{P}^N$, we have  
  either $\bar{D}\in |\Delta_0+\Gamma|$, $N=d+3$, or $\bar{D}=\Delta_0$, $N=d+5$.

\end{proof}
Now we continue the classification by considering $pr_c|_{\bar{X}}$. First we consider the cases where $c\in \bar{X}$.

Case I) $\bar{X}$ is $\mathbb{P}^2$ and $\pi^*(K_X+\Lambda)=\mathcal{O}_{\mathbb{P}^2}(1)$. Then $\bar{\Lambda}+\bar{D}\in |\mathcal{O}_{\mathbb{P}^2}(4)|$ and  $pr_c|_{\bar{X}}$ induces a fibration over $\mathbb{P}^1$.
Since $pr_c|_{\bar{X}}=\Phi_{\pi^*H^0(X,K_X+\Lambda)}$ restricting to $\bar{D}$ factors through an involution $\tau$, there are four possibilities:
\begin{enumerate}[1)]
	\item $ \bar{D}\in  |\mathcal{O}_{\mathbb{P}^2}(4)|$, $\bar{\Lambda}=0$ and $c\not \in \bar{D}$.
	\item $ \bar{D}\in  |\mathcal{O}_{\mathbb{P}^2}(3)|$, $\bar{\Lambda}\in  |\mathcal{O}_{\mathbb{P}^2}(1)|$ and $c\in \bar{D}$.
	\item $\bar{\Lambda},\bar{D}\in  |\mathcal{O}_{\mathbb{P}^2}(2)|$ and $c\not \in \bar{D}$.
	\item $\bar{D}\in  |\mathcal{O}_{\mathbb{P}^2}(1)|$, $\bar{\Lambda}\in  |\mathcal{O}_{\mathbb{P}^2}(3)|$ and $c\in \bar{D}$.
\end{enumerate}
The cases 2) and 4) can be excluded as $\bar{\Lambda}\cdot\bar{D}$ must be even by observation of case 8) in \textsection\ref{clssofslc}.
Examples of 1) and 3) will be given in Example \ref{examp1} and Example \ref{examp2}.

Case II) $\bar{X}$ is $C_{N-1}$ and $c$ is the vertex. The involution $\tau$ on $\bar{D}$ forces $\bar{\Lambda}$ to be two rulings and $\bar{D}\in |\mathcal{O}_{C_{N-1}}(2)|$.

Case III) $\bar{X}$ is $C_{N-1}$ and $c$ is not the vertex. We show that this does not occur by contradiction. By
Cor \ref{lineIntersecting} $pr_c|_{\bar{X}}$ is an isomorphism outside the ruling $l$ passing through $c$. The involution $\tau$ forces $\bar{D}$ to be $l$. Now $\bar{D}$ intersects $\bar{\Lambda}$ at the vertex of $C_{N-1}$ and two other points which may coincide.
However these points could not be glued into a Gorenstein slc singularity as in \textsection\ref{clssofslc}. Hence we obtain a contradiction.

Case IV) $\bar{X}$ is $\Sigma_{d}$. We show that this does not occur as well. Similarly as in Case III), we see that  $\bar{D}$ should be a ruling $\Gamma$ passing through $c$. Hence $\bar{D}\cdot \bar{\Lambda}=3$, which is impossible similarly  as discussed in Case I).

Case V) $\bar{X}$ is $\mathbb{P}^2$ embedding in $\mathbb{P}^5$. We show that this does not occur. Since by Cor \ref{lineIntersecting}, we see that $pr_c|_{\bar{X}}$ is an isomorphism outside $c$. However $pr_c|_{\bar{X}}$ restricting to $\bar{D}$ should factor through $\tau$, which is impossible.

Next we consider the cases where $c\not\in \bar{X}$. We see $pr_c|_{\bar{X}}$ is a morphism of degree at most two by Cor \ref{lineIntersecting} and it contracts no curve on $\bar{X}$.

Case VI) $pr_c|_{\bar{X}}$ is a double cover. In this case $\bar{X}$ is a quadric in $\mathbb{P}^3$. Denote  $pr_c|_{\bar{X}}$ by $\delta$. The branch curve $B\in |\mathcal{O}_{\mathbb{P}^2}(2)|$. $K_{\bar{X}}+\bar{D}+\bar{\Lambda}=\mathcal{O}_{\bar{X}}(1)=\delta^*\mathcal{O}_{\mathbb{P}^2}(1)$. Hence
$\bar{D}+\bar{\Lambda}\in |\delta^*\mathcal{O}_{\mathbb{P}^2}(3)|=|\mathcal{O}_{\bar{X}}(3)|$. Since $pr_c|_{\bar{X}}$ restricting to $\bar{D}$ factors through an involution $\tau$,
$\bar{D}$ must be the pre-image of a curve $C$ on $\mathbb{P}^2$ under $pr_c|_{\bar{X}}$. Denote by $m$  the degree of $C$. $m\le 3$.  

We have the following commutative diagram:
 \[
 \xymatrix{
 \bar{D}\ar@{^{(}->}[r]
 &\bar{X}\ar[d]^{\pi}\ar[drr]^{pr_c|_{\bar{X}}}\ar@{^{(}->}[rr]^{\Phi_{|K_{\bar{X}}+\bar{D}+\bar{\Lambda}|}} && \mathbb{P}^3 \ar@{-->}[d]^{pr_c}\\
 & X \ar[rr]_{\Phi_{|K_{X}+\Lambda|}} && \mathbb{P}^2
 }.
 \]
 We see that $X$ is also a double cover of $\mathbb{P}^2$ induced by $\Phi_{|K_{X}+\Lambda|}$. The branch curve is $2C+B$. 
 $\Lambda\in |\Phi_{|K_{X}+\Lambda|}^*\mathcal{O}_{\mathbb{P}^2}(3-m)|$.

If $pr_c|_{\bar{X}}$ is birational, it must be a normalization map. Its image must be $X$ and $\pi=pr_c|_{\bar{X}}$.
Then by Cor \ref{proj-image}, $(\bar{X},\bar{D},\bar{\Lambda})$ is one of the followings:

Case VII) $\bar{X}$ is $\mathbb{P}^2$ embedding in $\mathbb{P}^5$. $\bar{D}\in  |\mathcal{O}_{\mathbb{P}^2}(1)|$.  $\bar{\Lambda}\in  |\mathcal{O}_{\mathbb{P}^2}(4)|$.

Case VIII) $\bar{X}$ is $C_{N-1}$, $N\ge 4$. $\bar{D}$ is a sum of two rulings. $\bar{\Lambda}$ is linearly equivalent to $2H_{\infty}$ where $H_{\infty}$ is a hyperplane section.

Case IX) $\bar{X}$ is $\Sigma_{d}$ embedded in $\mathbb{P}^N$ , $N\ge 4$.
Either $N=d+5$, $\bar{D}=\Delta_0$ and $\bar{\Lambda}\in |2\Delta_0+(2d+4)\Gamma|$,
or $d=1$, $N=4$, $\bar{D}\in |\Delta_0+\Gamma|$ which is irreducible and $\bar{\Lambda}\in |2\Delta_0+4\Gamma|$ (we see that $\bar{D}$ would not be $\Delta_0+\Gamma$ as $\Delta_0\cap\Gamma$ could not be glued into a Gorenstein slc singularity as in \textsection\ref{clssofslc}. Hence the $d\ge 2$ cases and $d=1$, $\bar{D}=\Delta_0+\Gamma$ case are excluded).

We summarize the above results in the following theorem:
\begin{theorem}\label{delta1,0}
Let $(X,\Lambda)$ be an irreducible non-normal Gorenstein stable log surface as before. Assume $\Delta(X,K_X+\Lambda)=1$ and $\Delta(\bar{X},\pi^*(K_X+\Lambda))=0$.
Then there are the following possibilities:

\begin{itemize}
  \item $\bar{X}$ is $\mathbb{P}^2$. $\bar{\Lambda}\in |\mathcal{O}_{\mathbb{P}^2}(2)|$ and $\bar{D}\in |\mathcal{O}_{\mathbb{P}^2}(2)|$. Moreover, $c\not\in \bar{D}$. ($p_g(X,\Lambda)=2$)
  \item $\bar{X}$ is $\mathbb{P}^2$. $\bar{\Lambda}=0$ and $\bar{D}\in |\mathcal{O}_{\mathbb{P}^2}(4)|$. Moreover, $c\not\in \bar{D}$. ($p_g(X,\Lambda)=2$)
  \item $\bar{X}$ is $C_{N-1}$. $\bar{\Lambda}$ is two rulings and $\bar{D}\in |\mathcal{O}_{C_{N-1}}(2)|$. $c$ is the vertex of $C_{N-1}$. ($p_g(X,\Lambda)=N$)

  \item $\bar{X}$ is a quadric in $\mathbb{P}^3$. $\bar{D}\in |\mathcal{O}_{\bar{X}}(m)|$ and $\bar{\Lambda} \in |\mathcal{O}_{\bar{X}}(3-m)|$.

        Moreover, $X$ is a double cover of $\mathbb{P}^2$ induced by $\Phi_{|K_{X}+\Lambda|}$. The branch curve is $2C+B$, where $C$,$B$ are reduced curves of degree $m$, $2$. $\Lambda\in |\Phi_{|K_{X}+\Lambda|}^*\mathcal{O}_{\mathbb{P}^2}(3-m)|$.  $0<m \le 3$. ($p_g(X,\Lambda)=3$)
  \item $\bar X$  is a Veronese embedding of $\mathbb{P}^2$.
        $\bar{D}\in |\mathcal{O}_{\mathbb{P}^2}(1)|$ and $\bar{\Lambda}\in |\mathcal{O}_{\mathbb{P}^2}(4)|$. $X$ is the projection image of $\bar X$. ($p_g(X,\Lambda)=5$)
  \item $\bar X$  is  $C_{N-1}\subset\mathbb{P}^N$, $N\ge 4$.
       $\bar{D}$ is a sum of two lines passing through the vertex, and $\bar{\Lambda}\in |\mathcal{O}_{C_{N-1}}(2)|$. $X$ is the projection image of $\bar X$. ($p_g(X,\Lambda)=N$)
  \item $\bar X$  is $\Sigma_d$ embedded in $\mathbb{P}^N$, $N\ge 4$. 
        We have either $N=d+5$, $\bar{D}=\Delta_0$ and $\bar{\Lambda}\in |2\Delta_0+(2d+4)\Gamma|$;
        or $d=1$, $N=4$, $\bar{D}\in |\Delta_0+\Gamma|$ which is irreducible and $\bar{\Lambda}\in |2\Delta_0+4\Gamma|$. $X$ is the projection image of $\bar X$. ($p_g(X,\Lambda)=N$)

\end{itemize}
\end{theorem}

\begin{example}\label{examp1}
Let $\bar{D}$ be a smooth quadric in $\mathbb{P}^2$ defined by $x^2+y^2+z^2=0$. $\tau$ acts on $\bar{D}$ by $x\mapsto -x, y \mapsto -y, z\mapsto z$. 
Let $\bar{\Lambda}=L_{x=0}+L_{y=0}$, where $L_{x=0}$, $L_{y=0}$ are two lines.
Gluing $(\mathbb{P}^2, \bar{\Lambda})$ along $\bar{D}$ by $\tau$ we obtain a non-normal Gorenstein stable log surface $(X,\Lambda)$ with $(K_X+\Lambda)^2=p_g(X,\Lambda)-1=1$.
\end{example}
\begin{example}\label{examp2}
Let $\bar{D}$ be a smooth quartic in $\mathbb{P}^2$ defined by $x^4+y^4+z^4=0$. $\tau$ acts on $\bar{D}$ by $x\mapsto -x, y \mapsto -y, z\mapsto z$. 
Gluing $\mathbb{P}^2$ along $\bar{D}$ by $\tau$ we obtain a non-normal Gorenstein stable surface $X$ with $K_X^2=p_g-1=1$.
\end{example}

\begin{cor}\label{irrnonnormalstabledelta1}
  Let $X$ be an irreducible non-normal Gorenstein stable surface with $\Delta(X,K_X)=1$. Then $X$ is one of the followings:
  \begin{itemize}

    \item $X$ is a double cover of $\mathbb{P}^2$. The branch curve is $2C+B$, where $C$ and $B$ are reduced curves of degree $4-k$ and $2k$ respectively. $k=2,3$.  ($p_g(X)=3$)
    \item $X$ is obtained from a log surface $(\bar{X},\bar{D})$ by gluing the 2-section $\bar{D}$,   where $(\bar{X},\bar{D})$ is a normal Gorenstein stable log surface as in Thm \ref{delta-genus-one} (2) or (3).  ($p_g(X)=2$)
    \item $X$ is obtained from $\mathbb{P}^2$ by gluing a quartic curve $\bar{D}$. ($p_g(X)=2$)
    \item  $X$ is obtained from a quadric in $\mathbb{P}^3$ by gluing a curve $\bar{D}$, where $\bar{D}\in |\mathcal{O}_{\bar{X}}(3)|$. ($p_g(X)=3$)
  \end{itemize}

\end{cor}

\section{reducible Gorenstein stable log surfaces with $\Delta(X,K_X+\Lambda)=1$}
\label{reducible-stable}

In this section we consider reducible Gorenstein stable log surfaces. They are glued by irreducible ones along the connecting curves.

\begin{lemma}\label{restsecions}
Let $X$ be a connected $S_2$ scheme of pure dimension, $\mathcal{L}$ be an invertible sheaf such that $\dim \mathrm{Bs} |\mathcal{L}|<\dim X-1$ and $C$ be a subscheme of codimension 1. Then
\begin{align*}
d_{X\to C}(\mathcal{L})=\dim <\Phi_{|\mathcal{L}|}(C)>+1,
\end{align*}
where $<\Phi_{|\mathcal{L}|}(C)>$ is the projective subspace of $|\mathcal{L}|^*$ spanned by $\Phi_{|\mathcal{L}|}(C)$.
\end{lemma}
\begin{proof}
Write $\mathbb{P}:=|\mathcal{L}|^*$.
We have the following commutative diagram:
\[ \xymatrix{
      H^0(\mathbb{P},\mathcal{O}_{\mathbb{P}}(1))\ar[rr]^<(0.2){\mathcal{R}_{\mathbb{P}\to \Phi_{|\mathcal{L}|}(C)}}\ar[d]^{\cong}_{\Phi_{|\mathcal{L}|}^*} & & H^0(\Phi_{|\mathcal{L}|}(C),\mathcal{O}_{\mathbb{P}}(1)|_{\Phi_{|\mathcal{L}|}(C)})\ar@{^{(}->}[d]_{\Phi_{|\mathcal{L}|}|_{C}^*}  \\
    H^0(X,\mathcal{L})\ar[rr]^{\mathcal{R}_{X\to C}}& & H^0(C,\mathcal{L}|_C)
   }.
\]
Therefore $d_{X\to C}(\mathcal{L})=d_{\mathbb{P}\to \Phi_{|\mathcal{L}|}(C)}(\mathcal{O}_{\mathbb{P}}(1))=\dim <\Phi_{|\mathcal{L}|}(C)>+1$.
\end{proof}

Still we call a curve $C$ on a demi-normal scheme $X$ a \emph{line} if its proper transformation $\bar{C}$ is a line on the normalization $\bar{X}$ of $X$.

\begin{lemma}\label{restdim}
  Let $(X,\Lambda)$ be an irreducible Gorenstein stable log surface with a reduced curve $C$ on it. Then
   \begin{itemize}
     \item[(i)]  if $X$ is normal and $\Delta(X,K_X+\Lambda)=0$, then $d_{X\to C}(K_X+\Lambda)\ge2$.
         Moreover '=' holds if and only if $C$ is a line on $X$.
     \item[(ii)]  if $X$ is normal and $\Delta(X,K_X+\Lambda)=1$, then $d_{X\to C}(K_X+\Lambda)\ge1$.

     Moreover, '=' holds if and only if $|K_X+\Lambda|$ is composite with a pencil and $C$ is a fiber on $X$.

     If $|K_X+\Lambda|$ is not composite with a pencil, then $d_{X\to C}(K_X+\Lambda)\ge2$ and '=' holds if and only if $\Phi_{|K_X+\Lambda|}(C)$ is a line in $|K_X+\Lambda|^*$.

     \item[(iii)] if $X$ is non-normal and $\Delta(X,K_X+\Lambda)=1$, then $d_{X\to C}(K_X+\Lambda)\ge1$.
     Moreover, '=' holds if and only if $|K_X+\Lambda|$ has a base point, and $C$ is a line  passing through the base point.

   \end{itemize}
\end{lemma}
\begin{proof}
  (i) and (ii) follow from   Cor \ref{gepg-2}, Thm \ref{delta-genus-one} and Lemma \ref{restsecions}.

  For (iii), we note that $ \pi^* H^0(X,K_X+\Lambda)$ is of codimension 1 in $H^0(X,\pi^*(K_X+\Lambda))$. Hence it has a base point $c$ in $ \mathbb{P}:=|\pi^*(K_X+\Lambda)|^*$, and $\pi^* H^0(X,K_X+\Lambda)$ corresponds to the space of hyperplanes  of $ \mathbb{P}$ passing through  $c$. 
  Since $H^0(X,K_X+\Lambda)\cong \pi^* H^0(X,K_X+\Lambda)$, we have $|K_X+\Lambda|^* \cong \mathbb{P}(\pi^* H^0(X,K_X+\Lambda)^*)$.
  Denote by $pr_c\colon \mathbb{P}\dashrightarrow \mathbb{P}(\pi^* H^0(X,K_X+\Lambda)^*)$ the projection from the point $c$.

  We have the following commutative diagram:
  \[
  \xymatrix{
  \bar{C}\ar@{^{(}->}[r]\ar[d] & \bar{X}\ar@{^{(}->}[rrr]^{\Phi_{|\pi^*(K_X+\Lambda)|}}\ar@{-->}[rrrd]^{\Phi_{\pi^*H^0(X,K_X+\Lambda)}}\ar[d]_{\pi}& & & \mathbb{P}\ar@{-->}[d]^{pr_c}\\
  C\ar@{^{(}->}[r] &X\ar@{-->}[rr]_{\Phi_{|K_X+\Lambda|}} && |K_X+\Lambda|^* \ar[r]^(.35){\cong}  &\mathbb{P}(\pi^* H^0(X,K_X+\Lambda)^*).
  }
  \]

We regard $\Phi_{|\pi^*(K_X+\Lambda)|}$ as an inclusion.
By Lemma \ref{restsecions}, $d_{X\to C}(K_X+\Lambda)\ge1$. If '=' holds, $\Phi_{|K_X+\Lambda|}(C)$ is a point. Then $pr_c(\bar{C})$ is a point,  which implies $\bar{C}$ is a line passing through $c$. Therefore $C$ is a line  passing through $\pi(c)$, which is the base point of $|K_X+\Lambda|$.
\end{proof}

In the following, we consider reducible stable log surfaces, for example, a reducible stable log surface $(X,\Lambda)$ with $X=X_1\cup X_2$ and $C:=X_1\cap X_2$ the connecting curve. We use the partial normalization  $\nu \colon \bar{X}=\bar{X_1}\sqcup\bar{X_2}\to X$ along $C$. Since $X_i\cong \bar{X_i}$, we do not distinguish between them.
$\nu^*(K_X+\Lambda)|_{\bar{X_i}}$ is denoted by $(K_X+\Lambda)|_{X_i}$.
\begin{lemma}
\label{2comps}
  Let $(X,\Lambda)$ be a connected Gorenstein stable log surface.
  Assume $X=X_1\cup X_2$, where $X_1$ is connected and $X_2$ is irreducible.
  Denote by $C:=X_1\cap X_2$  the connecting curve of $X_1$ and $X_2$.

  Then:
\begin{itemize}
\item[(i)]  $\Delta(X,K_X+\Lambda)\ge \Delta(X_1,(K_X+\Lambda)|_{X_1})$.

\item[(ii)] if '=' holds and $\Delta(X_2,(K_X+\Lambda)|_{X_2})\le 1$,
  then
  \begin{itemize}
    \item either $X_2$ is non-normal and $\Delta(X_2,(K_X+\Lambda)|_{X_2})=1$. $d_{X_2\to C}((K_X+\Lambda)|_{X_2})=1$. 
        $C$ is a line on $X_2$  passing through the base point of $|(K_X+\Lambda)|_{X_2}|$. Moreover, $d_{X_1\to C}((K_X+\Lambda)|_{X_1})\le 1$.
    \item  $X_2$ is normal and $\Delta(X_2,(K_X+\Lambda)|_{X_2})=0$, $|(K_X+\Lambda)|_{X_2}|$ is very ample. $d_{X_2\to C}((K_X+\Lambda)|_{X_2})=2$. $C$ is a line on $X_2$. Moreover, $d_{X_1\to C}((K_X+\Lambda)|_{X_1})\le 2$;
    \item or $X_2$ is normal with $\Delta(X_2,(K_X+\Lambda)|_{X_2})=1$, and $|(K_X+\Lambda)|_{X_2}|$ is composite with a pencil of elliptic curves. $d_{X_2\to C}((K_X+\Lambda)|_{X_2})=1$. $C$ is a fiber on $X_2$. Moreover, $d_{X_1\to C}((K_X+\Lambda)|_{X_1})\le 1$.

  \end{itemize}
\end{itemize}

\end{lemma}
\begin{proof}
  $(K_X+\Lambda)^2=(K_X+\Lambda)|_{X_1}^2+(K_X+\Lambda)|_{X_2}^2$ together with  Remark \ref{fibersection} gives
\begin{equation}\label{deltaIneq}
\begin{split}
\Delta(X,K_X+\Lambda) & \ge \Delta(X_1,(K_X+\Lambda)|_{X_1})+\Delta(X_2,(K_X+\Lambda)|_{X_2})\\
&+\max\{d_{X_1\to C}((K_X+\Lambda)|_{X_1}),d_{X_2\to C}((K_X+\Lambda)|_{X_2})\}-2.
\end{split}
\end{equation}
  By Lemma \ref{restdim} we see that $\Delta(X_2,(K_X+\Lambda)|_{X_2})+d_{X_2\to C}((K_X+\Lambda)|_{X_2})\ge 2$.  Thus $\Delta(X,K_X+\Lambda)  \ge \Delta(X_1,(K_X+\Lambda)|_{X_1})$.

  If '=' holds, then
  \begin{align*}
  \max\{d_{X_1\to C}((K_X+\Lambda)|_{X_1}),d_{X_2\to C}((K_X+\Lambda)|_{X_2})\}= 2- \Delta(X_2,(K_X+\Lambda)|_{X_2}).
  \end{align*}
  Therefore, by Lemma \ref{restdim}, $d_{X_2\to C}((K_X+\Lambda)|_{X_2})=2$ if $\Delta(X_2,(K_X+\Lambda)|_{X_2})=0$, and $d_{X_2\to C}((K_X+\Lambda)|_{X_2})=1$ if $\Delta(X_2,(K_X+\Lambda)|_{X_2})=1$. The other statements of (ii) follow from Lemma \ref{restdim}.

\end{proof}

\begin{cor}\label{genus-like}
Let $(X,\Lambda)$ be a connected Gorenstein stable log surface.  Assume $Y\subset X$ is a connected subsurface. Then we have $\Delta(X,K_X+\Lambda)\ge \Delta(Y,(K_X+\Lambda)|_{Y})$.
\end{cor}
\begin{proof}
If $X_i\subset X$ is an irreducible surface connected to $Y$, $\Delta(Y,(K_X+\Lambda)|_{Y})\le \Delta(Y\cup X_i,(K_X+\Lambda)|_{Y\cup X_i})$ by Lemma \ref{2comps}. Then by induction hypothesis, we have $\Delta(Y,(K_X+\Lambda)|_{Y})\le \Delta(X,K_X+\Lambda)$.
\end{proof}

\begin{definition}
Let $(X,\Lambda)$ be a reducible Gorenstein stable log surface. Write $X=\bigcup X_i$. We say that a rational map $\Phi$ separates $X_i$ and $X_j$ if $\Phi(X_i\setminus X_i\cap X_j)\cap \Phi(X_j\setminus X_i\cap X_j)=\emptyset$.
\end{definition}

\begin{lemma}\label{globalsection}
Let $(X,\Lambda)$ be a connected reducible Gorenstein stable log surface with $X=X_1\cup X_2$ such that $X_1$ is connected and $X_2$ is irreducible. $i_1\colon X_1\hookrightarrow X$ and $i_2\colon X_2\hookrightarrow X$ are the inclusion maps.
Then
\begin{itemize}
\item[(i)] We have the following commutative diagram:
\[
  \xymatrix{
   0 \ar[r] & \ker i_1^* \ar[r]\ar[d]^{\cong} & H^0(X,K_X+\Lambda) \ar[r]^{i_1^*}\ar[d]^{i_2^*} &H^0(X_1,(K_X+\Lambda)|_{X_1})\ar[d]^{-\mathcal{R}_{X_1\to X_1\cap X_2}}\\
   0 \ar[r] & \ker\mathcal{R}_{X_2\to X_1\cap X_2} \ar[r] & H^0(X_2,(K_X+\Lambda)|_{X_2}) \ar[r]^{\mathcal{R}_{X_2\to X_1\cap X_2}} &H^0(X_1\cap X_2,(K_X+\Lambda)|_{X_1\cap X_2})
   .
  }
\]
\item[(ii)] If $\mathrm{im}\, \mathcal{R}_{X_1\to X_1\cap X_2}\subset \mathrm{im} \, \mathcal{R}_{X_2\to X_1\cap X_2}$, then $i_1^*$ is surjective.
\item[(iii)] If $(K_X+\Lambda)|_{X_2}$ is very ample and $X_1\cap X_2$ is a line on $X_2$, then $i_1^*$ is surjective and $\Phi_{|K_X+\Lambda)|}$ separates $X_1$ and $X_2$.
\item[(iv)] If each $(K_X+\Lambda)|_{X_i}$ is very ample and $X_1\cap X_2$ is a line on each $X_i$, then $K_X+\Lambda$ is very ample.

\end{itemize}

\end{lemma}
\begin{proof}
(i) and (ii) are obvious by chasing the diagram.

For (iii), first we see that $\mathcal{R}_{X_2\to X_1\cap X_2}$ is surjective.  Then any section in $H^0(X_1,(K_X+\Lambda)|_{X_1})$ can be extended into a section in  $H^0(X,K_X+\Lambda)$. Therefore $i_1^*$ is surjective. Next $\ker \mathcal{R}_{X_2\to X_1\cap X_2}$ is nontrivial and its base part is  the line $X_1\cap X_2$. Then $\ker i_1^*$ is nontrivial and its base part is $X_1$. Therefore $\Phi_{|K_X+\Lambda)|}$ separates $X_1$ and $X_2$.

For (iv) we see that $\mathcal{R}_{X\to X_i}$ is surjective and the kernel is nontrivial by (ii). Then we have plenty of sections to separate points and tangents, which implies $K_X+\Lambda$ is very ample.

\end{proof}
\begin{cor}\label{equaldelta}
  Let $(X,\Lambda)$ be a log surface as in Lemma \ref{2comps}.
  We assume further $X_1$ is irreducible and $\Delta(X,K_X+\Lambda)= \Delta(X_i,(K_X+\Lambda)|_{X_i})\le 1$ for $i=1,2$.
  Then
  \begin{itemize}
    \item[(i)] if $\Delta(X,K_X+\Lambda)=0$, then $X_1\cap X_2$ is a line on each $X_i$ and $K_X+\Lambda$ is very ample.

    \item[(ii)] if $\Delta(X,K_X+\Lambda)=1$,
        then
        \begin{itemize}
          \item either $X_1$ and $X_2$ are both normal. Each $|(K_X+\Lambda)|_{X_i}|$ is composite with a pencil. $X_1\cap X_2$ is a fiber on each $X_i$. $|K_X+\Lambda|$ is composite with a pencil as well.

          \item or $X_1$ and $X_2$ are both non-normal.
          The base points of $|(K_X+\Lambda)|_{X_i}|$ coincide into the unique base point of $|K_X+\Lambda|$. $X_1\cap X_2$ is a line passing through the base point of $|(K_X+\Lambda)|_{X_i}|$ on each $X_i$.
        \end{itemize}

  \end{itemize}
\end{cor}
\begin{proof}
(i) follows from Lemma \ref{2comps}  and Lemma \ref{globalsection}.

For (ii), applying Lemma \ref{2comps} we see that either $X_i$ is non-normal or $(K_X+\Lambda)|_{X_i}$ is composite with a pencil. It is easy to see that either $X_1$, $X_2$ are both non-normal or $(K_X+\Lambda)|_{X_1}$ and $(K_X+\Lambda)|_{X_2}$ are both composite with a pencil of elliptic curves, since the geometric genus of $X_1\cap X_2$ on $X_1$ or $X_2$ should coincide.

If $(K_X+\Lambda)|_{X_1}$ and  $(K_X+\Lambda)|_{X_2}$ are both composite with a pencil of elliptic curve, then by Lemma \ref{2comps} each $(K_X+\Lambda)|_{X_i}$ is composite with a pencil and $X_1\cap X_2$ is a fiber on each $X_i$. Therefore $|K_X+\Lambda|$ is composite with a pencil as well. 

If $X_1$ and $X_2$ are both non-normal,
then by Lemma \ref{restdim}, $X_1\cap X_2$ is a line passing through the base point of $|(K_X+\Lambda)|_{X_i}|$ on each $X_i$. These two base points must coincide, otherwise no nontrivial section in $H^0(X_i,(K_X+\Lambda)|_{X_i})$ can be glued into a section of $H^0(X,K_X+\Lambda)$.

\end{proof}
\begin{lemma}
\label{decrease1}
  Let $(X,\Lambda)$ be a connected Gorenstein stable log surface which has two irreducible components $X_1$, $X_2$.
  Assume further $\Delta(X_1,(K_X+\Lambda)|_{X_1})=\Delta(X_2,(K_X+\Lambda)|_{X_2})=0$ and $\Delta(X,K_X+\Lambda)=1$.
  Then
   \[d_{X_1\to X_1\cap X_2}((K_X+\Lambda)|_{X_1})=d_{X_2\to X_1\cap X_2}((K_X+\Lambda)|_{X_2})=3.\]
  Moreover, $(X_i,X_1\cap X_2+\Lambda_i)$ is as in Lemma \ref{d=3}

\end{lemma}
\begin{proof}
  By (\ref{deltaIneq}), we have $\max\{d_{X_1\to X_1\cap X_2}((K_X+\Lambda)|_{X_1}),d_{X_2\to X_1\cap X_2}((K_X+\Lambda)|_{X_2})\}\le 3$. Moreover $d_{X_i\to X_1\cap X_2}((K_X+\Lambda)|_{X_i})\ge 2$ by Lemma \ref{restdim}. Thus there must be one $d_{X_i\to X_1\cap X_2}((K_X+\Lambda)|_{X_i})=3$.

  Now we suppose $d_{X_1\to X_1\cap X_2}((K_X+\Lambda)|_{X_1})=2$ and $d_{X_2\to X_1\cap X_2}((K_X+\Lambda)|_{X_2})=3$ for a contradiction.
  Then $X_1\cap X_2$ is a line on $X_1$ and not a line on $X_2$. Thus nonzero elements of $\mathrm{im}\,\mathrm{Res}_{X_i\to X_1\cap X_2}$ have different degrees. Hence only those sections in $H^0(X_i,(K_X+\Lambda)|_{X_i})$ vanishing on $X_1\cap X_2$ can be glued together into a global section.
  Therefore
  \begin{align*}
  p_g(X,\Lambda)&=\dim\ker \mathcal{R}_{X_1\to X_1\cap X_2}+\dim \ker \mathcal{R}_{X_2\to X_1\cap X_2}\\
  &\le h^0(X_1,(K_X+\Lambda)|_{X_1})-2+h^0(X_2,(K_X+\Lambda)|_{X_2})-2\\
  &\le (K_X+\Lambda)^2,
  \end{align*}
  a contradiction. This completes the proof.

\end{proof}

\begin{lemma}\label{d=3}
Let $(X,D+\Lambda)$ be a normal Gorenstein stable log surface with $\Delta(X,K_X+D+\Lambda)=0$ and $d_{X\to D}(K_X+D+\Lambda)=3$. Then $(X,D+\Lambda)$ is one of the following:

\begin{enumerate}[(i)]
\item $X$ is $\mathbb{P}^{2}$, $D\in|\mathcal{O}_{\mathbb{P}^{2}}(k)|$ and $\Lambda\in|\mathcal{O}_{\mathbb{P}^{2}}(4-k)|$ ($2\le k \le 4$) ;
\item $X$ is $\mathbb{P}^{2}$, $D\in|\mathcal{O}_{\mathbb{P}^{2}}(1)|$  and $\Lambda\in|\mathcal{O}_{\mathbb{P}^{2}}(4)|$;
\item $X$ is $\Sigma_d$, $D\in|\Delta_0+\Gamma|$ and $N=d+3$ and  $\Lambda\in|2\Delta_0+(2d+2)\Gamma|$;
\item $X$ is $\Sigma_d$, $D=\Delta_0$ and $N=d+5$ and  $\Lambda\in|2\Delta_0+(2d+4)\Gamma|$;

\item $X$ is a singular quadric $C_2$ in $\mathbb{P}^{3}$, $D\in |\mathcal{O}_{X}(1)|$ and  $\Lambda\in|\mathcal{O}_{X}(2)|$. Moreover, $D$ does not pass the singularity of $X$;
\item  $X$ is a cone $C_{N-1}\hookrightarrow \mathbb{P}^{N}$, $D$ is a sum of two different rulings, and $\Lambda\in |\mathcal{O}_{X}(2)|$. ($N\ge 3$)

\end{enumerate}
	
\end{lemma}
\begin{proof}
First by Cor \ref{gepg-2}, $\Delta(X,K_X+D+\Lambda)=0$ implies $X$ is embedded in $\mathbb{P}^{N}$ by $\Phi:=\Phi_{|K_X+D+\Lambda|}$ as a surface of minimal degree and $D+\Lambda$ is as described in Cor \ref{gepg-2}. Second $d_{X\to D}(K_X+D+\Lambda)=3$ implies the embedding image $\Phi(D)$ of $D$ spans a projective plane $<\Phi(D)>\cong \mathbb{P}^{2}$ in $\mathbb{P}^{N}$.

If $X$ is $\mathbb{P}^{2}$ embedded in $\mathbb{P}^{2}$, $D$ and $\Lambda$ is clear.
If $X$ is $\mathbb{P}^{2}$ embedded in $\mathbb{P}^{5}$, then by Thm \ref{linearsectionthm} we see that  $\Phi(D)$ is quadric in $<\Phi(D)>$ (taking a generic line $l\in <\Phi(D)>$, we have $\mathrm{length}~ l\cap \Phi(X)=\mathrm{length}~ l\cap \Phi(D)\le 2$). Hence $\Phi(D)\cdot \mathcal{O}_{\mathbb{P}^{N}}(1)=2$. Then $D.(K_X+D+\Lambda)=2$. Therefore $D\in|\mathcal{O}_{\mathbb{P}^{2}}(1)|$ and $\Lambda\in|\mathcal{O}_{\mathbb{P}^{2}}(4)|$. If $X$ is $\Sigma_d$, similarly we have $D.(K_X+D+\Lambda)=2$. Hence either $D\in|\Delta_0+\Gamma|$ and $N=d+3$ which is the case (iii),  or $D=\Delta_0$ and $N=d+5$  which is the case (iv).
If is a $X$ is a cone $C_{N-1}$, $N\ge 3$, similarly we have $D.(K_X+D+\Lambda)=2$. Hence $D$ is either a sum of two different rulings which is the case (vi), or $N=3$ and $D$ is a hyperplane section of $X$ such that the vertex $v\not\in D$ which is the case (v).

\end{proof}

\begin{cor}\label{diffdelta}
Let $(X,\Lambda)$ be a connected Gorenstein stable log surface.
  Assume $X=X_1\cup X_2$, where $X_i$ is irreducible.
  Assume further $\Delta(X,K_X+\Lambda)= \Delta(X_1,(K_X+\Lambda)|_{X_1})=1$ and
  $\Delta(X_2,(K_X+\Lambda)|_{X_2})=0$.

  Then
  \begin{itemize}
    \item [(i)] either each $(K_X+\Lambda)|_{X_i}$ is very ample and  the curve $X_1\cap X_2$ is a line on each $X_i$.
        Moreover, $K_X+\Lambda$ is very ample; 
    \item [(ii)] or $X_1$ is non-normal.
    $X_1\cap X_2$ is a line on each $X_i$.

  \end{itemize}

\end{cor}
\begin{proof}
  First by Lemma \ref{2comps}, $X_1\cap X_2$ will be a line on $X_2$ and $d_{X_1\to X_1\cap X_2}((K_X+\Lambda)|_{X_1})=1$, or $2$.

  For the case $X_1$ is normal, we claim that $d_{X_1\to X_1\cap X_2}((K_X+\Lambda)|_{X_1})\not =1$. Otherwise, by Lemma \ref{restdim}, $|(K_X+\Lambda)|_{X_1}|$ is composite with a pencil of elliptic curves and the connecting curve $X_1\cap X_2$ will be an elliptic fiber on $X_1$. Hence $\mathrm{im}\,\mathcal{R}_{X_1\to X_1\cap X_2}$ consists of constant sections. Thus only those sections of $H^0(X_i, (K_X+\Lambda)|_{X_i})$ vanishing on $X_1\cap X_2$ can be glued together into a global section. Therefore,
  \begin{align*}
  p_g(X,\Lambda)&=\dim\ker \mathcal{R}_{X_1\to X_1\cap X_2}+\dim \ker \mathcal{R}_{X_2\to X_1\cap X_2}\\
  &=h^0(X_1,(K_X+\Lambda)|_{X_1})+h^0(X_2,(K_X+\Lambda)|_{X_2})-4\\
  &< (K_X+\Lambda)^2,
  \end{align*}
  a contradiction.

  Hence we have $d_{X_1\to X_1\cap X_2}((K_X+\Lambda)|_{X_1})=2$. Therefore $(K_X+\Lambda)|_{X_1}$ is not composite with a pencil and the log canonical image of $X_1\cap X_2$ is a line.
  Next we show that $X_1$ is not a double cover of $\mathbb{P}^2$.  Otherwise, on $X_1$, $X_1\cap X_2$ is a double covering curve of $\mathbb{P}^1$.
  Then nonzero sections in the image of $\mathcal{R}_{X_i\to X_1\cap X_2}$ have different degrees. Hence only those sections of $H^0(X_i, (K_X+\Lambda)|_{X_i})$ vanishing on $X_1\cap X_2$ can be glued together into a global section.
  Thus we will obtain a contradiction as before.

  Therefore $(K_X+\Lambda)|_{X_1}$ is very ample. $X_1\cap X_2$ is a line on $X_1$. The other statements follow.

  For the non-normal case,
  we take a partial normalization $\nu\colon \tilde{X}\to X$ along the non-normal locus on $X_1$.
  $\tilde{X}=\tilde{X_1}\cup X_2$, where $\tilde{X_1}$ is the normalization of $X_1$.
  It is easy to check that $\Delta(\tilde{X},\nu^*(K_X+\Lambda))=0$.
  Then we see that $X_1\cap X_2$ is a line on each $X_i$.
  Hence we complete the proof.

\end{proof}
\begin{example}
  Let $X_1$ be a $\mathbb{P}^2$ blown up at $k$ distinct points as in Thm \ref{delta-genus-one} (7). Each $E_i$ is a line with respect to $-K_{X_1}$. Let $\Lambda_1$ be a general element of the form $E_1+...+E_k+B\in |-2K_{X_1}|$, then $\Lambda_1$ is nodal and  $E_i$ intersects $B$ at 3 distinct points.
  Let $(X_2,\Lambda_2)$ be a log surface as in Cor \ref{gepg-2}(iii) such that  $\Lambda_2=\Gamma_1+...+\Gamma_s+D$, where $\Gamma_i$ is a ruling and it intersects $D$  at 3 distinct points.
  Let $\tau\colon E_i\cong \Gamma_j$ be an isomorphism mapping  $E_i\cap B$ to $\Gamma_j\cap D$.

  Then  we can glue $(X_1 \cup X_2, \Lambda_1+\Lambda_2)$ along $E_i$ and $\Gamma_j$ by $\tau$ to obtain a log surface as in Thm \ref{diffdelta} (i).
\end{example}

The following theorem  and corollary  are first obtained in \cite{LR13}.

\begin{thm}\label{thm: log noether} (Stable log Noether inequality, cf. \cite[Thm 4.1]{LR13})
 Let $(X,\Lambda)$ be a connected Gorenstein stable log surface.
 Then
 \[\Delta(X,K_X+\Lambda)\ge 0. \]

\end{thm}
\begin{proof}
This follows directly by Cor \ref{gepg-2}, Lemma \ref{nonnorm&irred} and Lemma \ref{2comps}.

\end{proof}

\begin{cor}\label{cor: nonnormal equality} (cf. \cite[Cor 4.10]{LR13})
 Let $(X, \Lambda)$ be a Gorenstein stable log surface such that
 $\Delta(X,K_X+\Lambda)=0$.
 Write $X=\bigcup X_i$, where $X_i$ is an irreducible component.

 Then
\begin{enumerate}
\item $\Delta(X_i, (K_X+\Lambda)|_{X_i})=0$. 
\item $K_X+\Lambda$ is very ample. It defines an embedding $\phi\colon X\hookrightarrow \mathbb{P}=|K_X+\Lambda|^*$;
\item if $X_{i}\cap X_{j}\not = \emptyset $ in codimension 1, then $X_{i}\cap X_{j}$ is a line on both $X_i$ and $X_j$;
\item $X$ is a tree of $X_i$ glued along lines.
\end{enumerate}
In particular, $\Lambda\neq 0$.
\end{cor}

Finally we are able to classify reducible Gorenstein stable log surfaces with $\Delta(X,K_X+\Lambda)=1$.
\begin{thm}\label{nonnormdelta-1}
Let $(X,\Lambda)$ be a reducible Gorenstein stable log surface with $\Delta(X,K_X+\Lambda)=1$.
Write $X=\bigcup X_i$, where $X_i$ is an irreducible component.

Then
$X$ has a unique minimal connected component $U$ such that $\Delta(U,(K_X+\Lambda)|_{U})=1$, $X\setminus U$ is composed of several trees of surfaces $T_j$'s with $\Delta(T_j,(K_X+\Lambda)|_{T_j})=0$ and $X$ is glued by $U$ and $T_j$'s along lines, i.e. $X=U\cup \bigcup T_j$ with each $U\cap T_j$ a line with respect to $K_X+\Lambda$.

$U$ is one of the followings:
        \begin{itemize}
        \item[(i)] $X=U$ is a string of surfaces $X_i$ with $\Delta(X_i,(K_X+\Lambda)|_{X_i})=1$ and $|(K_{X}+\Lambda)|_{X_i}|$  composite with a pencil of elliptic curves. The connecting curves are all fibers.
            Moreover, in this case $|K_X+\Lambda|$ is composite with a pencil of elliptic curves.

        \item[(ii)]  $U$ is a string of surfaces glued along lines whose two end surfaces $X_i$'s are non-normal with $\Delta(X_i,(K_X+\Lambda)|_{X_i})=1$.

            Moreover, if $U$ is reducible, $|K_X+\Lambda|$ has a unique base point $c\in X$ and all the connecting curves of $U$  pass through $c$.
            
        \item[(iii)]   $U$ is irreducible and $|(K_{X}+\Lambda)|_{U}|$ is very ample.  
            Moreover, in this case $K_X+\Lambda$ is very ample.
        \item[(iv)]   $U=X_j\cup X_k $ where $X_j$, $X_k$ are two irreducible surfaces with  $\Delta(X_j,(K_X+\Lambda)|_{X_j})=\Delta(X_k,(K_X+\Lambda)|_{X_k})=0$. All the connecting curves of $X$ are lines except $X_j\cap X_k$. 

        \item[(v)]  $U$ is a cycle of surface $X_i$ with $\Delta(X_i,(K_X+\Lambda)|_{X_i})=0$. All the connecting curves of $X$ are lines.  

   \end{itemize}
\end{thm}
\begin{proof}
First by Cor \ref{genus-like} we see each component $X_i$ has $\Delta(X_i,(K_X+\Lambda)|_{X_i})=1$ or $0$. Second, once we obtain a connected component $U\subset X$ such that $\Delta(U,(K_X+\Lambda)|_{U})=1$ and $\Delta(X_i,(K_X+\Lambda)|_{X_i})=0$ for each $X_i\subset X\setminus U$,
we see that  $X\setminus U$ will be composed of several trees of surfaces $T_j$'s with $\Delta(T_j,(K_X+\Lambda)|_{T_j})=0$ by Lemma \ref{2comps} and induction hypothesis. Therefore it remains to describe $U$.

We distinguish between two cases whether there exists an irreducible component $X_i$ such that $\Delta(X_i,(K_X+\Lambda)|_{X_i})=1$.

Case 1.  There exists a component $X_i$ such that $\Delta(X_i,(K_X+\Lambda)|_{X_i})=1$.
 Then  by Cor \ref{equaldelta} and
Cor \ref{diffdelta}, either $(K_X+\Lambda)|_{X_i}$ is composite with a pencil of elliptic curves, $X_i$ is non-normal, or   $(K_X+\Lambda)|_{X_i}$ is very ample.

If $(K_X+\Lambda)|_{X_i}$ is composite with a pencil of elliptic curves,
then by Cor \ref{equaldelta} and
Cor \ref{diffdelta}, for each irreducible components $X_j$ connected to $X_i$, $(K_X+\Lambda)|_{X_j}$ is composite with a pencil of elliptic curves as well, $X_i\cap X_j$ is a fiber and $\Delta(X_j,(K_X+\Lambda)|_{X_j})=1$.
Therefore, inductively, each $(K_X+\Lambda)|_{X_k}$ is composite with a pencil of elliptic curves. 
Moreover, each $X_i$ is connected to at most two other components, as there are at most two fibers as the connecting curves which pass through the base point of $|(K_X+\Lambda)|_{X_i}|$  on $X_i$.
Thus $X=U$ is a string of such surfaces glued along fibers.

If $(K_X+\Lambda)|_{X_i}$ is very ample, then by Cor \ref{equaldelta} and Cor \ref{diffdelta},   every other irreducible component $X_j$ connected to $X_i$ has $\Delta(X_j,(K_X+\Lambda)|_{X_j})=0$ and the connecting curves are lines.
Hence we see that $U=X_i$.
Moreover, $K_X+\Lambda$ is very ample by Lemma \ref{globalsection}.

If $X_i$ is non-normal, then by Lemma \ref{2comps}
other component $X_j$ is either non-normal or has $\Delta(X_j,(K_X+\Lambda)|_{X_j})=0$ and $X$ is a tree of such surfaces.
Let $\nu\colon \tilde{X}\to X$ be a partial normalization along the non-normal curves of $X_j$' s.
It is easy to check that $\Delta(\bar{X}, \nu^*(K_X+\Lambda))=0$. $\nu^*H^0(X,K_X+\Lambda)\subset H^0(\bar{X}, \nu^*(K_X+\Lambda))$ has a base point $c\in \mathbb{P}:=|\nu^*(K_X+\Lambda)|^*$.
We see that $c$ is also the base point of $|(K_X+\Lambda)|_{X_j}|$ for any non-normal $X_j$.
If there is only one non-normal surface $X_i$, then $U=X_i$. If there are two non-normal surfaces $X_i$, $X_j$, then $U$ is the unique string of surfaces contained in $X$ with $X_i$, $X_j$ as the ends. If there are more than two non-normal surfaces, we claim that they lie on a unique minimal string of surfaces, which is $U$. Otherwise there are  three non-normal surfaces contained in a fork of surfaces $Y\subset X$. By Lemma \ref{2comps}, we see that the base point $c$ lies on each $X_k\subset Y$. However, in the central surface of $Y$, there are three connecting curves passing through $c$, which is impossible. Hence the claim is true. 

Case 2. Each irreducible component $X_i$ has $\Delta(X_i,(K_X+\Lambda)|_{X_i})=0$.

If there are two irreducible components $X_j$, $X_{k}$ such that
$\Delta(X_j\cup X_k,(K_X+\Lambda)|_{X_j\cup X_k})=1$,
then by Lemma \ref{decrease1}, $X_j\cap X_k$ is neither a line on $X_j$ nor $X_k$, and  $d_{X_j\to X_j\cap X_k}((K_X+\Lambda)|_{X_j})=d_{X_k\to X_j\cap X_k}((K_X+\Lambda)|_{X_k})=3$.
Furthermore, it is easy to see that $X$ is a tree of surfaces. Ungluing $X$ along $X_j\cap X_{k}$, there will be two connected trees of surfaces $V_1\supset X_j$, $V_2\supset X_{k}$. We claim $\Delta(V_1,(K_X+\Lambda)|_{V_1})=\Delta(V_2,(K_X+\Lambda)|_{V_2})=0$. Otherwise assume $\Delta(V_1,(K_X+\Lambda)|_{V_1})=1$, then $\Delta(V_1\cup X_{k},(K_X+\Lambda)|_{V_1\cup X_{k}})=1$ implies $ V_1\cap X_{k}=X_j\cap X_{k}$ is a line on $X_k$ by Lemma \ref{2comps}, a contradiction.
Hence $U=X_j\cup X_k$ and $X$ is a tree of surfaces whose connecting curves are lines except $X_j\cap X_k$.

Finally we consider the case that all the connecting curve are lines. It is easy to see that $X$ is not a tree, otherwise $\Delta(X,K_X+\Lambda)=0$. Hence there is a minimal cycle of surfaces $U$ such that $\Delta(U,(K_X+\Lambda)|_U)=1$. 

\end{proof}

\begin{cor}
Let $X$ be a connected reducible Gorenstein stable surface  with $K_X^2= p_g-1$.
Then $X$ is a union of two $\mathbb{P}^2$'s glued along a quartic curve on each $\mathbb{P}^2$. Moreover, $X$ is a double cover of $\mathbb{P}^2$ and can be deformed into a smooth stable surface  with $K_X^2= p_g-1$. ($K_X^2=2$)

\end{cor}
\begin{proof}

If $X\setminus U$ is not empty, then there is a leaf surface $X_i$ such that $\Delta(X_i,K_X|_{X_i})=0$. On the other hand, we have $K_X|_{X_i}=K_{X_i}+\bar{D_i}$ and the connecting curve $\bar{D_i}$ is  a line, which is impossible by Cor \ref{gepg-2}.
Hence $X=U$. A similar discussion tells us that $X$ is not a string of $X_i$ such that $|K_X|_{X_i}|$ is composite with a pencil of elliptic curves.
Therefore $X$ is either a string of surfaces containing non-normal ones or a union of two irreducible components whose connecting curve is not a line.
We show that the first case does not occur. Otherwise $X=U$ is a string of surfaces glued along lines passing through $c\in X$ where $c$ is the base point of $|K_X|$. However, in this case $c$ would not be a Gorenstein slc singularity  of type 7) in  \textsection \ref{clssofslc} , a contradiction.

Therefore $X=X_1\cup X_2$ with $\Delta(X_i, K_X|_{X_i})=0$ and $d_{X_i\to X_1\cap X_2}(K_X|_{X_i})=3$. Then $X_i$ is $\mathbb{P}^2$, $\Sigma_d$ or $C_{N-1}$.
While in these cases $\ker \mathcal{R}_{X_i\to X_1\cap X_2}=0$ (otherwise $K_X|_{X_i}=K_{X_i}+X_1\cap X_2\ge X_1\cap X_2$ as divisors, which is impossible),
thus $h^0(X_i,K_X|_{X_i})=d_{X_i\to X_1\cap X_2}(K_X|_{X_i})=3$.
Hence $X_i$ is $\mathbb{P}^2$, and the connecting curve $X_1\cap X_2$ is a curve of degree 4.

It is easy to see that $X$ is defined by the equation $z^2=f_4^2(x_0,x_1,x_3)$, where $f_4$ is the defining equation of the quartic curve. When deforming $f^2_4$ into a generic equation of degree 8, $X$ is deformed into a smooth stable surface  with $K_X^2= p_g-1$.
\end{proof}

Summing up the results of Gorenstein stable surfaces, we obtain

\begin{thm}
Let $X$ be a connected  Gorenstein stable surface  with $K_X^2= p_g-1$.

Then it is one of the followings:
\begin{itemize}
	 \item $X$ is a double cover of $\mathbb{P}^2$.  The branch curve $B\in |\mathcal{O}_{\mathbb{P}^2}(8)|$ is a reduced curve which admits curve singularities of lc double-covering type.  ($p_g(X)=3$)
	 \item $X$ is canonically embedded as a hypersurface of degree 10 in the smooth locus of $\mathbb{P}(1,1,2,5)$. ($p_g(X)=2$)
	 \item $X$ is a double cover of $\mathbb{P}^2$. The branch curve is $2C+B$, where $C$ and $B$ are reduced curves of degree $4-k$ and $2k$ respectively. $k=2,3$.  ($p_g(X)=3$)
	 \item $X$ is obtained from a log surface $(\bar{X},\bar{D})$ by gluing the 2-section $\bar{D}$,   where $(\bar{X},\bar{D})$ is a normal Gorenstein stable log surface as in Thm \ref{delta-genus-one} (2) or (3).  ($p_g(X)=2$)
	 \item $X$ is obtained from $\mathbb{P}^2$ by gluing a quartic curve. ($p_g(X)=2$)
	 \item  $X$ is obtained from a quadric in $\mathbb{P}^3$ by gluing a curve $\bar{D}$, where $\bar{D}\in |\mathcal{O}_{\bar{X}}(3)|$. ($p_g(X)=3$)
	 \item $X$ is a union of two $\mathbb{P}^2$'s glued along a quartic curve on each $\mathbb{P}^2$.  ($p_g(X)=3$)
\end{itemize}
\end{thm}

\begin{rmk}
	We have confirmed the conjecture in \cite{LR13} that there are no Gorenstein stable surfaces with $K_X^2= p_g-1\ge 3$.
	
	We also notice that as the $\Delta$-genus grows bigger, the stable surfaces become more difficult to classify as polarized varieties with big $\Delta$-genus are hard to classify in general.

\end{rmk}

\end{document}